\documentclass[12pt,a4paper]{amsart}
\usepackage{a4wide}
\usepackage{amsfonts,amsthm,amsmath,amssymb,amscd}
\usepackage{graphicx}
\usepackage[usenames,dvipsnames]{color}
\usepackage{pstricks}

\usepackage{enumerate}
\theoremstyle{definition}
\newtheorem{dfn}{Definition}
\newtheorem{rem}{Remark}
\newtheorem*{xrem}{Remark}
\theoremstyle{plain}
\newtheorem{lem}[dfn]{Lemma}
\newtheorem{twr}[dfn]{Theorem}
\newtheorem{stw}[dfn]{Proposition}

\newtheorem{cor}[dfn]{Corollary}
\newtheorem{fact}{Fact}

\def\Re{{\mathrm{Re}}}

\numberwithin{equation}{section}

\title[Indecomposable continua]{ Indecomposable continua in exponential dynamics-Hausdorff dimension}

\thanks{A. Zdunik was supported in part by the Polish NCN grant  N N201 607940.}
\keywords{Holomorphic dynamics,  Julia sets, indecomposable continua, Hausdorff dimension}
\subjclass{Primary 37F35, Secondary  37B45}

\begin{document}

\author{{\L}ukasz Pawelec}
\address{{\L}ukasz Pawelec, Department of Mathematics and Mathematical Economics, Warsaw School of Economics, 
al.~Niepodleg\l{}o\'{s}ci~162, 02-554 Warszawa, Poland}
\email{LPawel@sgh.waw.pl}
\author{Anna Zdunik}
\address{Anna Zdunik,Institute of Mathematics, University of Warsaw,
ul.~Banacha~2, 02-097 Warszawa, Poland}
\email{A.Zdunik@mimuw.edu.pl}

\begin{abstract}
We study some forward invariant  sets appearing in the dynamics of the exponential family. 
We prove that the Hausdorff dimension of the sets under consideration is not larger than $1$.
This allows us to prove, as a consequence, a result for some dynamically defined  indecomposable continua which appear in the dynamics of the exponential family. We prove that the Hausdorff dimension of these continua  is equal to one.
\end{abstract}
\maketitle

\section{Introduction}\label{intro}
The dynamics of transcendental meromorphic functions has been developed as a counterpart to the theory of rational iterations.
A model family $\lambda\mapsto\lambda\exp(z)$ has been a subject of a  particular interest, playing a  similar role
as the family of quadratic polynomials $z^2+c$ in the iterations of rational maps, see e.g. \cite{de1} for a review of results on this family.

There are several properties of the iterations of entire (and, more general, transcendental) maps which have no counterpart in the dynamics of rational maps.

Below, we recall  basic definitions and facts. 

Let $f:\mathbb{C}\to \hat{\mathbb{C}}$ be a transcendental meromorphic function. As usually,
we shall denote by $f^n$ the $n$-th iterate of $f$: $f^n=f\circ f\circ \dots \circ f$. The  sequence $f^n(z)_{n=0}^\infty$ is called the trajectory of $z$.
The Fatou set $F(f)$ consists of all points $z\in\mathbb{C}$ for which there exists a neighbourhood $U\ni z$ such that the family of iterates $f^n_{|U}$ is defined in $U$ and forms a normal family.
The complement of $F(f)$ is called the Julia set of $f$ and it is denoted by $J(f)$.  An intuitive characterisation of the Julia set says that it carries the chaotic part of the dynamics. See e.g.  
 \cite{bergweiler} and~\cite{ku} for a detailed presentation of the theory. 

For the particular family of maps $f_\lambda(z)= \lambda e^z$ the structure of the Julia set and the dynamics have been studied intensively. 
In 1981 M. Misiurewicz answered the old  question (formulated by Fatou) and proved that the Julia set of the map $f_1(z)=e^z$ is the whole plane (\cite{misiurewicz}). A striking result, proved independently by M. Lyubich \cite{lyubich}  and M. Rees \cite{rees} says, that, nevertheless, the map is not ergodic with respect to the two- dimensional Lebesgue measure, and for Lebesgue almost every point $z$  the $\omega$-limit set  of $z$ (i.e. the set of accumulation points of the trajectory) is just the trajectory of the  singular value $0$ plus the point at infinity.

On the other hand, there is an open set of parameters $\lambda$ for which there exists a periodic attracting orbit (i.e. a point $p_\lambda$ such that $f^k_\lambda(p_\lambda)=p_\lambda$ and $|(f^k_\lambda)'(p_\lambda)|<1$ for some $k\in \mathbb{N}$). In this case, the Fatou set is nonempty and it contains a neighbourhood of this periodic attracting orbit. Actually, it turns out that in this case the Fatou set is just the  basin of attraction of this periodic orbit, i.e.
$$F(f_\lambda)=\{z\in\mathbb{C}:\lim_{n\to\infty}f^{nk}_\lambda(z)= f^i (p_\lambda){\rm ~for ~some~} i\in \{0,1,\dots {k-1\}\}}.$$
Moreover, this basin of attraction must contain the whole trajectory  of the singular value $0$. 
For instance, if $\lambda\in (0,\frac{1}{e})$, then $f_\lambda$ has a real attracting fixed point $p_\lambda$. For these values of $\lambda$  the Fatou set of $f_\lambda$ 
is connected and simply connected and its complement $J(f_\lambda)$ is a "Cantor bouquet" of curves. See \cite{de1} or \cite{schleicherparadox} for a survey of results in this direction. 

\

A general classification theorem for the components of the Fatou set leads to the following corollary:  if the parameter $\lambda$ is chosen so that   $f^n_\lambda(0)\to\infty$ then  the Julia set of $f_\lambda$  is the whole plane: $J(f_\lambda)=\mathbb{C}$ (see \cite{eremenkolyubich} or \cite{goldbergkeen}).

In this paper, we shall always assume that the parameter $\lambda$ is chosen so that  $f_\lambda^n(0)\to\infty$ sufficiently fast. Thus, in particular, for all maps $f_\lambda$ considered in this paper we have  $J(f_\lambda)=\mathbb{C}$ (see Section \ref{statement} for setting of precise assumptions).

\

We shall explain now the motivation of the present paper.

It is known that the set of escaping points 
$$I(f_\lambda)=\{z:f^n_\lambda(z)\to\infty\}$$
can be described in terms of so-called "dynamic rays" (see,  \cite{schleicherzimmer} and \cite{rempetopo}). Each  dynamic ray is a curve 
 $g_{\underline s}:(t_{\underline s},\infty)\to\mathbb{C}$.
  Moreover, $\mathrm{Re}(g_{\underline s}(t))\to +\infty$ as $t\to +\infty$. 
  Each such curve is characterised by an infinite sequence  $\underline{s}=(s_0,s_1,\dots, )$ of integers. The sequence $\underline s$ is frequently called the "external address" of the curve $g_{\underline s}$.
The dynamical meaning of the sequence $\underline{s}$ is the following: 
Let us divide the plane $\mathbb{C}$ into horizontal strips 
\begin{equation}\label{strips}
P_k=\{z\in\mathbb{C}:(2k-1)\pi-\mathrm{Arg}(\lambda)<\mathrm{Im}(z)\le(2k+1)\pi-\mathrm{Arg}(\lambda)\}
\end{equation}

If $z$ is a point on the curve $g_{\underline s}$, $z=g_{\underline s}(t)$ with $t$ sufficiently large, 
then, for every $n\ge 0$, $f^n_\lambda(z)\in P_{s_n}$. 

The classification of escaping points (see \cite{schleicherzimmer}, Corollary 6.9) says that this family of curves almost exhausts the set $I(f_\lambda)$.  Namely, if $z\in I(f_\lambda)$ then either $z$ belongs to some curve $g_{\underline s}$, or $z$ is a landing point of some curve $g_{\underline s}$, or else the singular value $0$ escapes, $0=g_{\underline s}(t_0)$ for some $t_0>t_{\underline s}$, and $z$ is eventually mapped to the initial piece of the curve  $g_{\underline s}$, cut off by the point $0$, i.e.
$$f^n_\lambda(z)=g_{\underline s}(t')$$
for some $t_{\underline s}<t'<t_0$.
 
As an example, let us consider $\lambda=1$ and the dynamical ray corresponding to the sequence $\underline s =(0,0,0,\dots)$. This set is just the real line $\mathbb{R}$. The set 

$$\bigcup_{n\ge 0}f^{-n}_\lambda(\mathbb{R})\cap \{\mathrm{Im}(z)\in [0,\pi]\}$$

is a union of infinitely many arcs extending to infinity. It was observed first by R. Devaney (see 
\cite{de1}) that the closure of this set, after a natural compactification at  $\infty$, becomes an indecomposable continuum.
This continuum can be equivalently described as the set of points whose trajectories  remain in the strip
$\mathrm{Im}(z)\in [0,\pi]$.

Next, again for $\lambda=1$, R. Devaney and X. Jarque discovered the existence of indecomposable continua defined as an accumulation set, in the Riemann sphere,  of some dynamic rays.
More precisely, they considered  in \cite{devaneyjarque} the  rays of the form $g_{\underline s}$, where $\underline s=(t_{m_1},0_{n_1}, t_{m_2}, 0_{n_2}\dots)$, where $t_{m_j}$ are blocks of integers, of length $m_j$, with all digits $\le M$. If the blocks of zeroes $0_{n_j}$ are sufficiently long, then the set of accumulation points of the ray $g_{\underline s}$ becomes an indecomposable continuum containing $g_{\underline s}$.

A more general result appeared in \cite{rempenonlanding}. The author shows  the following: assume that the singular value $0$ of the map $f_\lambda$ is on a dynamic ray, or it is a landing point of such a ray. Then there are uncountably many dynamic rays $g$ whose accumulation set (in the Riemann sphere) is an indecomposable continuum containing $g$.

\

The present paper was motivated by the abovementioned examples.  Our goal was to give a bound for the Hausdorff dimension of  dynamically defined indecomposable continua  appearing in the exponential dynamics.
It was already observed  in \cite{devaneyetds} that M.Lyubich's result \cite{lyubich} implies that two- dimensional Lebesgue measure of the indecomposable continuum described in \cite{de1}, is equal to $0$. Actually, it is easy to see that the same remark applies to the examples considered in \cite {devaneyjarque}.

Our goal in this paper is to prove, for a class of dynamically defined indecomposable continua for the exponential family, that their Hausdorff dimension is equal to one, so it takes on the smallest possible value. This class of continua contains, in particular, the examples described in \cite{de1} and \cite{devaneyjarque}.
So, our result is an essential strengthening of the previously known  estimates.

Our results apply also to a subclass of the class of continua described in \cite{rempenonlanding}.
It is now natural to ask about the possible values of the Hausdorff dimension of other dynamically defined indecomposable continua (appearing  in the dynamics of exponential maps),  not covered by our considerations (in particular- all the continua described in \cite{rempenonlanding}).

\

Our paper is organised as follows. In Section~\ref{statement} we formulate the condition on the parameter $\lambda$ ("super-growing parameters", see Definition~\ref{defsuper}). Then we introduce and describe some forward invariant sets for the dynamics of the map $f_\lambda$, where $\lambda$ is
a super-growing parameter.

We formulate a result (Theorem~\ref{bigM}) on the upper bound of the Hausdorff dimension of these forward invariant sets, namely, that the Hausdorff dimension of such set is not larger than $1$. This is a result of independent interest,  useful for further applications.

The proof of Theorem~\ref{bigM} is divided into several steps, and it is presented in Sections~\ref{induced}, ~\ref{ppositive},~\ref{nnegative} and ~\ref{conclusion}.

In Section~\ref{wnioski} we show in detail how Theorem~\ref{bigM} can be used to estimate the dimension of the  particular, dynamically defined  indecomposable continua. In each case when the application of Theorem~\ref{bigM} is possible, we get the strongest possible 
upper bound  on the dimension, proving that the Hausdorff dimension of the continuum is equal to one.
See Theorems \ref{contdevaney} and~\ref{contrempe} for the precise statement. 
In addition, it is worth noting  that one- dimensional Hausdorff measure of an indecomposable continuum in the plane is not $\sigma$-finite, see Proposition~ \ref{sskoncz}, Section~\ref{hm}.

\

\noindent {\bf Acknowledgements.}
We thank Bogus{\l}awa Karpi\'nska and Krzysztof Bara\'nski for interesting discussions on the subject.
We are grateful to  Dierk Schleicher for his  valuable  comments, and for drawing our attention to possible further generalisation of the results.

\section{Statement of the technical result}\label{statement}

 Throughout the paper we will work with the function $f_\lambda(z) = \lambda e^z$ for which  the trajectory of the singular value $0$  tends to
 infinity exponentially fast. We keep the following definition and notation, introduced in \cite{urbanskizduniknonhyperbolic}. 
Let 

\begin{equation}\label{imre}
    \beta_n = f_\lambda^{n}(0),~~~ \alpha_n={\rm Re}\beta_n
\end{equation}

\begin{dfn}\label{defsuper}
We say that the parameter $\lambda$ is \emph{super-growing} if $\alpha_n\to \infty$ and there exists a constant $c>0$ such that, for all $n$ large enough, 

\begin{equation}\label{super}
\alpha_{n+1}\ge ce^{\alpha_n}
\end{equation}
\end{dfn}

Note that the above condition is equivalent to

\begin{equation}\label{super2}
|\beta_{n+1}|\ge |\lambda|\exp\left (\frac{c}{|\lambda|}|\beta_n|\right)
\end{equation}
and to
\begin{equation}\label{superreal}
\alpha_n\ge\frac{c}{|\lambda|}|\beta_n|
\end{equation}

 From now on, unless stated otherwise,  we  assume that the   parameter $\lambda$  satisfies  the super-growing condition  (\ref{super}), with the constant $c$.
We shall keep  the notation (\ref{imre}). To simplify the notation, we  write  $f(z)$ instead of $f_\lambda(z)$, if it does not lead to a confusion.

\begin{xrem} For example, the  condition (\ref{super}) is satisfied if $|\mathrm{Im}(\beta_n)|$ is bounded  and $f_\lambda^{n}(0) \to \infty$ as $n \to +\infty$. Moreover, 
it was proved in \cite{we} that the Hausdorff dimension of parameters $\lambda$ satisfying the super-growing condition is equal to two.
(See also \cite{forsterschleicher} for a detailed description of the structure of parameters $\lambda$ for which the trajectory of the singular value escapes to $\infty$.)
\end{xrem}

\

In order to formulate the main technical result, we introduce the following  definition.

\begin{dfn}\label{thin}
Let $W\subset\mathbb{C}$ be a closed set.
 Denote by $W(R)_+$ the intersection $W(R)_+=W\cap \{{\rm Re}z=R\}$ and, analogously,
$W(R)_-=W\cap \{{\rm Re}z=-R\}$. 
Let 
$$w(R)=\max({\rm diam} W(R)_+,{\rm diam} W(R)_-).$$
  We call the set $W$ \emph{thin} if 
there is a constant $K>0$ such that
\begin{equation}\label{stozek}
\frac{|z|}{|\Re (z)|+1}<K
\end{equation}
for all $z\in W$, and
\begin{equation}\label{thinn}
\lim_{R\to\infty}\frac{\log_{+} w(R)}{\log R} =0.
\end{equation}

\end{dfn}

\begin{rem}
Obviously, every horizontal strip
$$\{z\in\mathbb{C}:|\mathrm{Im}(z)|\le P\}$$
is a thin set.

Note also that the condition (\ref{stozek}) implies that there is a cone symmetric with respect to the real axis, with
the opening angle smaller that $\pi$ and $R_0>0$ such that  the intersection $W\cap |\{\mathrm{Re}(z)|\ge R_0\}$  is contained in this cone.
\end{rem}   

The following definition introduces the set which will be an object of our estimates.

\begin{dfn}\label{lambda}
Let $\lambda\in\mathbb{C}\setminus\{0\}$. Put $f=f_\lambda$.
 Let $W$ be a thin set. We define the set $\Lambda_W$ as
 $$\Lambda_W=\{z\in\mathbb{C}: f^n(z)\in W~{\rm for~~every~~} n\ge 0\}.$$
\end{dfn}

Note that the set $\Lambda_W$ is a forward invariant closed subset of $\mathbb{C}$. Forward invariant means that $f(\Lambda_W)\subset \Lambda_W$.

\begin{rem}
Since we do not assume  any dynamical condition on $W$, it may even happen that  the set $\Lambda_W$ is  
empty. However,  in all our applications (see Section~\ref{wnioski}) the set $\Lambda_W$ contains a non-trivial continuum, so it has Hausdorff dimension at least one.
\end{rem}

We shall prove the following.

\begin{twr}\label{bigM}
Let $\lambda$ be a super-growing parameter. Let $W$ be a thin set. Put $f=f_\lambda$ and let $\Lambda_W$ be the set defined in Definition~\ref{lambda}. 

Put
\begin{equation}
Y_M=\{z\in\mathrm{C}:|\mathrm{Re}(z)|\ge M\}
\end{equation}
Then 
\begin{equation}
\lim_{M\to\infty}\mathrm{dim}_H(\Lambda_W\cap Y_M)\le 1.
\end{equation}

\end{twr}

(Note that the limit exists because the function $M\mapsto \mathrm{dim}_H(\Lambda_W\cap Y_M)$ is non-increasing).

\

So, writing $\Lambda_W$ as a union  $\Lambda_W=\Lambda_{W,bd}\cup \Lambda_{W, ubd}$, where $\Lambda_{W,bd}$, $\Lambda_{W,ubd}$  denote the subset of points in $\Lambda_W$ with bounded (resp.: unbounded) trajectory, we have

\begin{cor}\label{ubd}
$$\dim_H(\Lambda_{W,ubd})\le 1.$$
\end{cor}
\begin{proof}
Obviously, for every $M>0$,  $$\Lambda_{W,ubd}\subset \bigcup_{n=0}^\infty f^{-n} (\Lambda_W\cap Y_M).$$
Since $f$ is an analytic non-constant map,  every Borel set $A$  has the same Hausdorff dimension as its preimage  $f^{-1}(A)$. Taking a countable union of sets of the same dimension does not increase the dimension. So,  Corollary follows immediately from Theorem ~\ref{bigM}.
\end{proof}

\

The proof of Theorem~\ref{bigM}  is split into several steps. In Section~\ref{induced} we claim the existence of the special induced map, and we formulate, in Proposition~\ref{induk}, some numerical estimates for this map. 
Next, the precise  definition of the map and the proof of the required estimates is presented in Section~\ref{ppositive} and \ref{nnegative}. In Section~\ref{conclusion} we use these estimates to conclude the proof of Theorem~\ref{bigM}.

\section{The induced map}\label{induced}

Obviously, to prove Theorem~\ref{bigM} it is enough to consider integer values of $M$. So, from now one we shall assume that  $M\in\mathbb{N}$.

In this section, we claim the existence of an auxiliary induced map, with certain properties, see Proposition~\ref{induk} below.

First, we cut each horizontal strip $P_k$  (see (\ref{strips})) into rectangles 
\begin{equation}\label{initial}
R^k_r=\{z \colon r\leq\mathrm{Re}(z) < r+1\}\cap P_k
\end{equation} 

For an arbitrary $M\in\mathbb{N}$   denote by $Z_M$ the family of all  rectangles $R^k_r$ intersecting  $W\cap Y_M$.

Note that our assumption  ~(\ref{thinn}) on $W$  implies that the number $n(r)$ of rectangles in $Z_M$ intersecting the lines $\mathrm{Re}(z)=\pm r$, satisfies
$$\log n(r)=o(\log(r)).$$

The map $F$ will be defined in the union $W_M$ of all rectangles intersecting the set $W\cap Y_M$:
$$W_M=\bigcup_ {R^k_r\in Z_M}R^k_r.$$

The set $W_M$ splits naturally into two subsets:

$$W_M^+=W_M\cap\{z:\mathrm{Re}(z)>0\},$$

$$W_M^-=W_M\cap\{z:\mathrm{Re}(z)<0\}.$$

The next Proposition summarizes the required properties of the induced map $F$.

\begin{stw}\label{induk} Let $W$ be a thin set. Let $\lambda$ be a supergrowing parameter; put $f=f_\lambda$.  For every $\delta >0$ there exist   $M\in \mathbb{N}$  
and   a map   $F$ defined in $W_M$,  with the following properties:
\begin{itemize}
\item{}$F$ is constructed with appropriate iterates of the map $f$ and for every rectangle $R_r^k\in Z_M$  $F_{|R_r^k}=f^n$ for some $n\in\mathbb{N}$.
\item{} \begin{equation}\label{pokr}
F(\Lambda_{W}\cap Y_M)\subset \Lambda_{W}\cap Y_M
\end{equation}
\item{} For every rectangle $R^k_r\in Z_M$ the set $R^k_r\cap F^{-1}(W_M)$ can be covered by a family  $\mathcal{F}^k_r$ of disjoint subsets of $R^k_r$ such that each set $Q\in\mathcal{F}^k_r$ is mapped by $F$ bijectively onto 
its image $F(Q)$, which is contained in some rectangle $\widehat{R}^m_s\in Z_M$.
Moreover, the holomorphic branch of $F^{-1}$ mapping $F(Q)$ back onto $Q$ is well defined in a neighbourhood of  the whole rectangle $\widehat R^m_s$ and 

\begin{equation}\label{maineq}
\sum_{Q\in {\mathcal F}^k_r}\sup_{w\in Q}|F'(w)|^{-(1+\delta)}<\frac{1}{2}
\end{equation}

\end{itemize}
\end{stw}

\section{Definition and the estimates for the map F in $W_M^+$}\label{ppositive}

Let $\lambda$ be a parameter satisfying the assumption of Proposition~\ref{induk}.
We use the simplified notation $f=f_\lambda$.
The map  $F$  is defined in $W_M^+$  simply as $F(x)=f(x)$. Thus, $F$ is one-to-one, if restricted to any
rectangle $R_r^k$.
\begin{lem}\label{4.1} For $M$ large enough the inclusion 
$$F(\Lambda_W\cap Y_M^+)\subset \Lambda_W\cap Y_M$$
holds.
\end{lem}
\begin{proof}
It is obvious that $f(\Lambda_W)\subset \Lambda_W$. So, $F(\Lambda_W\cap Y_M^+)\subset \Lambda_W$.

Let $z\in \Lambda_W\cap Y_M^+$. Since $f(z)\in W$, we have
$$K(|\Re(f(z))|+1)>|f(z)|=|\lambda|e^{\Re(z)}\ge|\lambda|e^M.$$
where $K$ is the constant coming from the definition of a thin set. 

Thus,
$$|\Re(F(z))|=|\Re(f(z))|\ge \frac{1}{K}|\lambda|e^M-\frac{1}{K}$$

Thus, if $M$ is large enough, we conclude that  $|\Re (F(z))|\ge M$.
\end{proof}
\

 Now, we turn to the  proof of (\ref{maineq}).
  More precisely, we shall estimate the sum in  (\ref{maineq}), corresponding to the rectangles $R^k_r\in Z_M$, for $r$ positive
(thus: larger than or equal to $M$). The estimate is straightforward  and  based on Lemma~\ref{lemri}
below.


\begin{lem} \label{lemri} Fix  some $\delta\in(0,1)$, and  put $\delta'=\frac{\delta}{2}$.
There exist  $\widehat C$ (independent of $\delta$)  and $M>0$  (depending on $\delta$) such that, for $r\geq M$, and every rectangle $R_r^k$
\begin{equation}
\sum_{\widehat{R}^l_s\cap f(R_r^k)\neq \emptyset} \sup_{w\in f^{-1} \left (\widehat{R}^l_s\right )\cap R_r^k} |f'(w)|^{-(1+\delta)}\leq \widehat Ce^{-r\delta'}
\end{equation}
where the summation runs over all rectangles $\widehat R^l_s\in Z_M$ intersecting $f(R_r^k)$.
\end{lem}
\begin{proof}
 We assume now that $M$ is so large that $w(s)<s^{\delta'}$ for all $s\ge M$.

If  $z\in R_r^k$, we have
\begin{equation}\label{poch}
|f'(z)|=|f(z)| = |\lambda|e^{\mathrm{Re}(z)}\geq |\lambda| e^r
\end{equation}

The number of rectangles  $\widehat{R}^l_s\in Z_M$ for which  $\widehat{R}^l_s\cap f(R_r^k)\neq \emptyset$ can be estimated (very roughly) by
$$\sum_{M\le s\le |\lambda| e^{r+1}} w(s)\le  \sum_{M\le s\le |\lambda| e^{r+1}}s^{\delta'}\le Ce^{(1+\delta')(r+1)}.$$
Thus,
$$\sum_{\widehat{R}^l_s\cap f(R_r^k)\neq \emptyset} \sup_{w\in f^{-1} \left (\widehat{R}^l_s\right )
\cap R_r^k} |f'(w)|^{-(1+\delta)}\le C e^{(1+\delta')(r+1)}| \lambda|^{-(1+\delta)}e^{-r(1+\delta)}=\widehat Ce^{-r\delta'}$$

where $C$ and $\widehat C$ are  constants independent of $M$ and $\delta$ (see Figure 1).
\end{proof}

Assuming that $M$ is so large that $\hat C e^{-M\delta'}<\frac{1}{2}$ and denoting by  ${\mathcal F}^k_r$ the family of all sets of the form $Q=f^{-1}(\widehat R^l_s)\cap R^k_r$,  we thus get   the inequality (\ref{maineq}):

\begin{equation}\label{positive}
\sum_{Q\in{\mathcal F}^k_r}\sup_{w\in Q}|F'(w)|^{-(1+\delta)}\le \hat C e^{-r\delta'}\le\frac{1}{2}
\end{equation}

\begin{xrem}
Note that, in this part of the proof we did not use the "super-growing" property.
 We used only the fact that the domain of the map is contained in ${\rm Re} z\ge M$ and that the set $W$ is thin.
It is worth to observe that  above calculation is close in spirit   to the argument contained in \cite{ka}.
\end{xrem}
\begin{figure}[h]
\scalebox{1} 
{
\begin{pspicture}(0,-1.9378124)(11.221875,1.91)
\psbezier[linewidth=0.04,arrowsize=0.05291667cm 2.0,arrowlength=1.4,arrowinset=0.4]{->}(3.66,0.76)(4.338992,1.244762)(5.5556498,1.2)(6.48,0.7)
\psline[linewidth=0.02cm,arrowsize=0.05291667cm 2.0,arrowlength=1.4,arrowinset=0.4]{->}(8.4,-1.9)(8.4,1.9)
\psline[linewidth=0.02cm,arrowsize=0.05291667cm 2.0,arrowlength=1.4,arrowinset=0.4]{->}(6.6,-0.3)(10.6,-0.3)
\pscircle[linewidth=0.04,dimen=outer](8.4,-0.3){1.0}
\pscircle[linewidth=0.04,dimen=outer](8.4,-0.3){1.4}
\psbezier[linewidth=0.04,linestyle=dashed,dash=0.16cm 0.16cm](6.6,0.18571429)(8.388235,-0.1)(9.058824,0.1)(10.4,0.7)
\psbezier[linewidth=0.04,linestyle=dashed,dash=0.16cm 0.16cm](6.6,-0.75714284)(8.0,-0.1)(10.0,0.1)(10.4,-0.5)
\usefont{T1}{ppl}{m}{n}
\rput(2.5392187,0.45){$R_r^k$}
\usefont{T1}{ptm}{m}{n}
\rput(5.021406,1.55){$f$}
\usefont{T1}{ptm}{m}{n}
\rput(10.591406,0.13){$W$}
\psline[linewidth=0.02cm,arrowsize=0.05291667cm 2.0,arrowlength=1.4,arrowinset=0.4]{->}(1.8,-1.9)(1.8,1.9)
\psline[linewidth=0.02cm,arrowsize=0.05291667cm 2.0,arrowlength=1.4,arrowinset=0.4]{->}(0.0,-0.3)(4.0,-0.3)
\psframe[linewidth=0.04,dimen=outer](2.45,0.19)(2.2,-0.06)
\psline[linewidth=0.01cm](9.48,0.32)(9.7,0.1)
\psline[linewidth=0.01cm](7.3,0.08)(7.44,-0.06)
\psline[linewidth=0.01cm](7.14,0.14)(7.4,-0.16)
\psline[linewidth=0.01cm](7.06,0.08)(7.38,-0.26)
\psline[linewidth=0.01cm](7.04,-0.08)(7.36,-0.42)
\psline[linewidth=0.01cm](7.04,-0.24)(7.26,-0.46)
\psline[linewidth=0.01cm](7.02,-0.4)(7.12,-0.5)
\psline[linewidth=0.01cm](9.36,0.3)(9.72,-0.06)
\psline[linewidth=0.01cm](9.26,0.28)(9.64,-0.12)
\psline[linewidth=0.01cm](9.3,0.08)(9.48,-0.1)
\usefont{T1}{ptm}{m}{n}
\rput(10.381406,-1.71){$f(R_r^k)$}
\psline[linewidth=0.03cm,arrowsize=0.05291667cm 2.0,arrowlength=1.4,arrowinset=0.4]{->}(10.32,-1.5)(9.36,-0.94)
\end{pspicture} 
}
\caption{$F$ in $W_M^+$}

\end{figure}
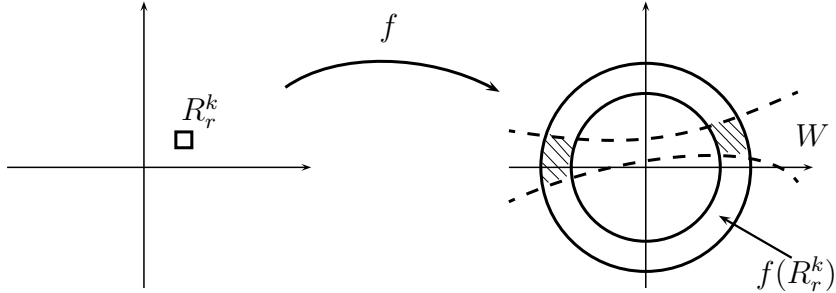
\section{Definition and the estimates for the map F in $W_M^-$.}\label{nnegative}
We keep the assumption that $\lambda$ is a supergrowing parameter and, as before, we abbreviate the notation writing $f:=f_\lambda$.

We start with the following easy observation.
\begin{stw}
If the singular value $0$ does not belong to the set $\Lambda_W$, then there exists $M>0$ such that the set $\Lambda_W$ is contained in the right half-plane ${\rm Re}z\ge -M$.
\end{stw}
\begin{proof}  If $0\notin\Lambda_W$ then there exists $k\ge 0$ such that $f^k(0)\notin W$, and, consequently, there exists  a ball $B(0,\eta)$
such that $f^k(z)\notin W$ for every $z\in B(0,\eta)$. Consequently, $B(0,\eta)\cap \Lambda_W=\emptyset$.
Thus, if ${\rm Re}w<\log \eta-\log|\lambda|$ then $w\notin \Lambda_W$ (since $f(w)\in B(0,\eta)$).
\end{proof}
Thus, the part of the proof contained in this section  is void in the case when $0\notin \Lambda_W$.
The definition of the map $F$ and the proof of  (\ref{maineq}) for $r$ negative is more involved and it uses the "super-growing" condition.

\

The proof of the following technical lemma is straightforward and left to the reader.
\begin{lem}\label{superepsilon}
Let $\lambda$ be a super-growing parameter, $\alpha_n={\rm Re}f^n_\lambda(0)$. Then for every $\varepsilon>0$ there exists $ N$ such that for all $n> N$
\begin{equation}\label{eps}
\alpha_1+\alpha_2+\dots+\alpha_n<\varepsilon\alpha_{n+1}
\end{equation}
\end{lem}

The Koebe distortion theorem implies that for every univalent map $f$ defined on some ball $B(z,r)$, the distortion of $f'$, restricted to the ball $B(z,\frac{r}{2})$ is
bounded by some constant, independent of the map. This constant will be denoted by $L$.
Recall the "super-growing" equivalent 
conditions (\ref{super}), (\ref{super2}) and (\ref{superreal}). Put
$$D=\frac{1}{4}\frac{c}{|\lambda|}\le\frac{1}{4}$$
where the constant $c$ comes from the super-growing conditions.

It follows from the condition 
(\ref{super2}) that there exists $l_0$ such that, for all $l\ge l_0$, the ball $B(\beta_l, 2D |\beta_l|)$
contains only one point of the trajectory of the point $0$. Put 
$$B_l:=B(\beta_l, D |\beta_l|)$$

Thus, the  inverse branches of $f^l$ are well defined in  $B_l$, with distortion bounded by $L$.

\

Let  $f^{-l}_0$ denote the branch of $f^{-l}$ following the backward trajectory $\beta_l, \beta_{l-1},\dots \beta_0=0$;  we put $\widetilde{B}_l=f^{-l}_0(B_l)$. Then $\widetilde {B}_l$ is a topological disc containing the point $0$.
Since $|(f^l)'(0)|=|\beta_1|\cdot\dots \cdot |\beta_l|$ and since the branch $f^{-l}_0$, sending $\beta_l$ to $0$,
is also well defined on the ball twice larger $2\cdot B_l=B(\beta_l, 2D|\beta_l|)$,
the set $\widetilde{B}_l$ is contained in the ball

\begin{equation}\label{balll}
B\left (0, \frac{L D|\beta_l|}{|\beta_1|\cdots|\beta_{l-1}||\beta_l|}\right )=B\left(\frac{LD}{|\beta_1|\cdots|\beta_{l-1}|}\right)
\end{equation}

 and contains the ball
\begin{equation}\label{ball}
B\left (0, \frac{\frac{1}{4}D|\beta_l|}{|\beta_1|\cdots|\beta_{l-1}||\beta_l|}\right )=B\left(0,\frac{\frac{1}{4}D}{|\beta_1|\cdots|\beta_{l-1}|}\right ).
\end{equation}

\
Next,  put 
$\widetilde G_l = f^{-1}(\frac{1}{e}\widetilde {B}_l)$.
where $\frac{1}{e} \tilde B_l$ is the image of $\tilde B_l$ under the rescaling $z\mapsto \frac{1}{e}z$.   So, $\tilde G_l$ is an unbounded set containing some left halfplane. Finally, put  $G_l = \widetilde G_l\cap W$  (see Figure 2).
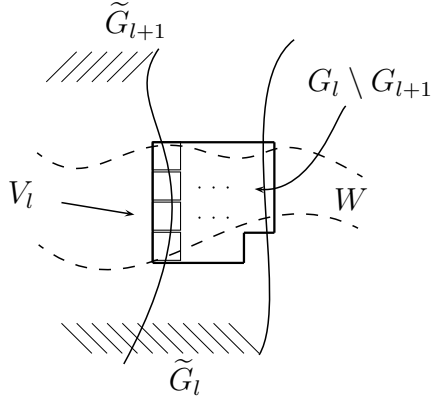
\begin{figure}[h]
\scalebox{1} 
{
\begin{pspicture}(0,-2.558125)(6.8228126,2.558125)
\psbezier[linewidth=0.02,linestyle=dashed,dash=0.16cm 0.16cm](0.6009375,0.5796875)(1.3698274,-0.0403125)(1.7444148,0.9596875)(2.7893164,0.5996875)(3.834218,0.2396875)(3.4409375,1.1196876)(4.87912,0.2596875)
\psbezier[linewidth=0.02,linestyle=dashed,dash=0.16cm 0.16cm](0.6209375,-0.6403125)(2.0609374,-1.7603126)(3.2209375,0.3996875)(4.8209376,-0.4203125)
\usefont{T1}{ptm}{m}{n}
\rput(4.7323437,-0.0303125){$W$}
\psbezier[linewidth=0.02](3.9809375,2.0796876)(3.1209376,1.3996875)(3.9609375,-1.3603125)(3.5209374,-2.1003125)
\psbezier[linewidth=0.02](2.2009375,1.9596875)(1.5609375,1.0396875)(3.2209375,0.5196875)(1.7409375,-2.2203126)
\usefont{T1}{ptm}{m}{n}
\rput(2.5523438,-2.39){$\widetilde{G}_l$}
\usefont{T1}{ptm}{m}{n}
\rput(1.9,2.33){$\widetilde{G}_{l+1}$}
\psline[linewidth=0.01cm](3.5209374,-2.0803125)(3.1209376,-1.6803125)
\psline[linewidth=0.01cm](3.3209374,-2.0803125)(2.9209375,-1.6803125)
\psline[linewidth=0.01cm](3.1209376,-2.0803125)(2.7209375,-1.6803125)
\psline[linewidth=0.01cm](2.9209375,-2.0803125)(2.5209374,-1.6803125)
\psline[linewidth=0.01cm](2.7209375,-2.0803125)(2.3209374,-1.6803125)
\psline[linewidth=0.01cm](2.5209374,-2.0803125)(2.1209376,-1.6803125)
\psline[linewidth=0.01cm](2.1209376,-2.0803125)(1.7209375,-1.6803125)
\psline[linewidth=0.01cm](2.3209374,-2.0803125)(1.9209375,-1.6803125)
\psline[linewidth=0.01cm](1.9209375,-2.0803125)(1.5209374,-1.6803125)
\psline[linewidth=0.01cm](1.5209374,-2.0803125)(1.1209375,-1.6803125)
\psline[linewidth=0.01cm](1.7209375,-2.0803125)(1.3209375,-1.6803125)
\psline[linewidth=0.01cm](2.1209376,1.9196875)(1.7209375,1.5196875)
\psline[linewidth=0.01cm](1.9209375,1.9196875)(1.5209374,1.5196875)
\psline[linewidth=0.01cm](1.7209375,1.9196875)(1.3209375,1.5196875)
\psline[linewidth=0.01cm](1.5209374,1.9196875)(1.1209375,1.5196875)
\psline[linewidth=0.01cm](1.3209375,1.9196875)(0.9209375,1.5196875)
\psline[linewidth=0.01cm](1.1209375,1.9196875)(0.7209375,1.5196875)
\psline[linewidth=0.01cm](0.9209375,-1.6803125)(1.3209375,-2.0803125)
\usefont{T1}{ptm}{m}{n}
\rput(0.39234376,0.0296875){$V_l$}
\psline[linewidth=0.02cm,arrowsize=0.05291667cm 2.0,arrowlength=1.4,arrowinset=0.4]{->}(0.9209375,-0.0803125)(1.8809375,-0.2403125)
\usefont{T1}{ptm}{m}{n}
\rput(4.9823437,1.5096875){$G_l\setminus G_{l+1}$}
\psbezier[linewidth=0.02,arrowsize=0.05291667cm 2.0,arrowlength=1.4,arrowinset=0.4]{->}(4.6609373,1.1996875)(4.552619,1.1241076)(4.2409377,0.0396875)(3.4609375,0.0996875)
\psframe[linewidth=0.0139999995,dimen=outer](2.4939375,0.7326875)(2.1079376,0.3466875)
\psframe[linewidth=0.0139999995,dimen=outer](2.4939375,0.3326875)(2.1079376,-0.0533125)
\psframe[linewidth=0.0139999995,dimen=outer](2.4939375,-0.0673125)(2.1079376,-0.4533125)
\psframe[linewidth=0.0139999995,dimen=outer](2.4939375,-0.4673125)(2.1079376,-0.8533125)
\psline[linewidth=0.03cm](2.1209376,0.7196875)(3.7209375,0.7196875)
\psline[linewidth=0.03cm](3.7209375,0.7196875)(3.7209375,-0.4803125)
\psline[linewidth=0.03cm](3.7209375,-0.4803125)(3.3209374,-0.4803125)
\psline[linewidth=0.03cm](3.3209374,-0.4803125)(3.3209374,-0.8803125)
\psline[linewidth=0.03cm](3.3209374,-0.8803125)(2.1209376,-0.8803125)
\psline[linewidth=0.03cm](2.1209376,-0.8803125)(2.1209376,0.7196875)
\psline[linewidth=0.03cm,linestyle=dotted,dotsep=0.16cm](2.7209375,-0.2803125)(3.1209376,-0.2803125)
\psline[linewidth=0.03cm,linestyle=dotted,dotsep=0.16cm](2.7209375,0.1196875)(3.1209376,0.1196875)
\end{pspicture} 
}
\caption{Definitions of sets $V_l$, $\widetilde{G}_l$, etc.}

\end{figure}

\begin{lem}
For all $l$ large enough,
${\rm cl} \widetilde G_{l+1}\subset\widetilde  G_l$. Consequently, ${\rm cl}  G_{l+1}\subset\ G_l$.
\end{lem}

\begin{proof}
Indeed, it follows from  (\ref{balll}) and (\ref{ball}) that, for $l$ large,  ${\rm cl} \widetilde B_{l+1}\subset \widetilde B_l$.
\end{proof}

Finally, set $V_l$ to be the union of all  rectangles $R^k_r\in Z_M$ which intersect the set  $G_l \setminus G_{l+1}$. Thus the sets $V_l$ are not disjoint,
but, for large $l$, one rectangle $R^k_r$ may intersect  two sets $V_l$ at most (see Figure 2).

Given $l_0\in\mathbb{N}$, we put   $M=[\alpha_{l_0}]+1$.
Note that, by the above estimates,  we have:
\begin{lem}\label{lem12}
If $l_0$ is large enough then
$$W_M^- \subset \bigcup_{l=l_0+1}^\infty V_l.$$
\end{lem}

\begin{proof}
Using the fact that $\widetilde{B}_l$ contains the ball $B(0, \frac{\frac{1}{4}D}{|\beta_1|\cdots|\beta_{l-1}|})$,
we conclude that the set  $\widetilde G_l$ contains the left halfplane

\begin{equation}\label{zaw}
\{\mathrm{Re}(z)<-\log 4 -1+\log D-(\alpha_{l-2}+\dots +\alpha_0 )-l\log|\lambda|\}
\end{equation}

 It follows from (\ref{zaw}) and Lemma~\ref{superepsilon} that, for large $l$,  $\widetilde G_l$ contains the left
halfplane 
$$\{{\rm Re}z\le-\alpha_{l-1}\}.$$

Thus, $$W_M^-=W\cap \{{\rm Re}z\le-([\alpha_{l_0}]+1)\}\subset W\cap \tilde G_{l_0+1}= G_{{l_0+1}}\subset\bigcup_{l={l_0+1}}^\infty V_l.$$
\end{proof}

\

Below, we  define the map $F$, separately on each set  $V_l$.

For $z\in V_l$ define $F(z) =   f \circ f^l \circ f = f^{l+2}$.
Note that the set $\tilde G_l$ is mapped by $f$ onto $\frac{1}{e}\widetilde B_l$.
Thus, $f$ maps $V_l$ into $\widetilde B_l$ and we have

\begin{equation}\nonumber
    V_l\stackrel{f}{\longrightarrow}\widetilde{B}_l\stackrel{f^l}{\longrightarrow} B_l\stackrel{f}{\longrightarrow} f(B_l).
\end{equation}
Note that the map is neither  onto nor injective.
Obviously,  the map $F{|_{V_l}}$ has a holomorphic extension to $\mathbb{C}$.
\begin{figure}[h]
\scalebox{1} 
{
\begin{pspicture}(0,-1.718125)(13.71,1.718125)
\pscustom[linewidth=0.02,fillstyle=solid,fillcolor=lightgray]
{
\newpath
\moveto(10.89,0.3596875)
\lineto(10.94,0.3396875)
\curveto(10.965,0.3296875)(10.99,0.2946875)(10.99,0.2696875)
\curveto(10.99,0.2446875)(10.995,0.1996875)(11.0,0.1796875)
\curveto(11.005,0.1596875)(11.03,0.1346875)(11.05,0.1296875)
\curveto(11.07,0.1246875)(11.095,0.0946875)(11.1,0.0696875)
\curveto(11.105,0.0446875)(11.11,-0.0053125)(11.11,-0.0303125)
\curveto(11.11,-0.0553125)(11.115,-0.1003125)(11.12,-0.1203125)
\curveto(11.125,-0.1403125)(11.14,-0.1853125)(11.15,-0.2103125)
\curveto(11.16,-0.2353125)(11.16,-0.2853125)(11.15,-0.3103125)
\curveto(11.14,-0.3353125)(11.12,-0.3853125)(11.11,-0.4103125)
\curveto(11.1,-0.4353125)(11.075,-0.4853125)(11.06,-0.5103125)
\curveto(11.045,-0.5353125)(11.01,-0.5753125)(10.99,-0.5903125)
\curveto(10.97,-0.6053125)(10.92,-0.6203125)(10.89,-0.6203125)
\curveto(10.86,-0.6203125)(10.805,-0.6203125)(10.78,-0.6203125)
\curveto(10.755,-0.6203125)(10.71,-0.6453125)(10.69,-0.6703125)
\curveto(10.67,-0.6953125)(10.625,-0.7153125)(10.6,-0.7103125)
\curveto(10.575,-0.7053125)(10.525,-0.7053125)(10.5,-0.7103125)
\curveto(10.475,-0.7153125)(10.43,-0.7303125)(10.41,-0.7403125)
\curveto(10.39,-0.7503125)(10.34,-0.7603125)(10.31,-0.7603125)
\curveto(10.28,-0.7603125)(10.23,-0.7603125)(10.21,-0.7603125)
\curveto(10.19,-0.7603125)(10.15,-0.7653125)(10.13,-0.7703125)
\curveto(10.11,-0.7753125)(10.06,-0.7553125)(10.03,-0.7303125)
\curveto(10.0,-0.7053125)(9.95,-0.6753125)(9.93,-0.6703125)
\curveto(9.91,-0.6653125)(9.865,-0.6653125)(9.84,-0.6703125)
\curveto(9.815,-0.6753125)(9.765,-0.6353125)(9.74,-0.5903125)
\curveto(9.715,-0.5453125)(9.68,-0.4753125)(9.67,-0.4503125)
\curveto(9.66,-0.4253125)(9.66,-0.3453125)(9.67,-0.2903125)
\curveto(9.68,-0.2353125)(9.69,-0.1253125)(9.69,-0.0703125)
\curveto(9.69,-0.0153125)(9.72,0.0746875)(9.75,0.1096875)
\curveto(9.78,0.1446875)(9.835,0.1996875)(9.86,0.2196875)
\curveto(9.885,0.2396875)(9.93,0.2646875)(9.95,0.2696875)
\curveto(9.97,0.2746875)(10.005,0.2996875)(10.02,0.3196875)
\curveto(10.035,0.3396875)(10.025,0.3646875)(10.0,0.3696875)
\curveto(9.975,0.3746875)(9.975,0.3846875)(10.0,0.3896875)
\curveto(10.025,0.3946875)(10.09,0.3796875)(10.13,0.3596875)
\curveto(10.17,0.3396875)(10.23,0.3346875)(10.25,0.3496875)
\curveto(10.27,0.3646875)(10.315,0.3996875)(10.34,0.4196875)
\curveto(10.365,0.4396875)(10.42,0.4646875)(10.45,0.4696875)
\curveto(10.48,0.4746875)(10.54,0.4746875)(10.57,0.4696875)
\curveto(10.6,0.4646875)(10.655,0.4596875)(10.68,0.4596875)
\curveto(10.705,0.4596875)(10.75,0.4546875)(10.77,0.4496875)
\curveto(10.79,0.4446875)(10.83,0.4296875)(10.85,0.4196875)
\curveto(10.87,0.4096875)(10.9,0.3896875)(10.93,0.3596875)
}
\pscustom[linewidth=0.02,fillstyle=solid]
{
\newpath
\moveto(10.49,0.3196875)
\lineto(10.45,0.3096875)
\curveto(10.43,0.3046875)(10.4,0.2846875)(10.39,0.2696875)
\curveto(10.38,0.2546875)(10.35,0.2346875)(10.33,0.2296875)
\curveto(10.31,0.2246875)(10.265,0.2146875)(10.24,0.2096875)
\curveto(10.215,0.2046875)(10.165,0.1996875)(10.14,0.1996875)
\curveto(10.115,0.1996875)(10.08,0.1796875)(10.07,0.1596875)
\curveto(10.06,0.1396875)(10.025,0.1196875)(10.0,0.1196875)
\curveto(9.975,0.1196875)(9.93,0.1146875)(9.91,0.1096875)
\curveto(9.89,0.1046875)(9.86,0.0796875)(9.85,0.0596875)
\curveto(9.84,0.0396875)(9.835,-0.0053125)(9.84,-0.0303125)
\curveto(9.845,-0.0553125)(9.85,-0.1053125)(9.85,-0.1303125)
\curveto(9.85,-0.1553125)(9.85,-0.2053125)(9.85,-0.2303125)
\curveto(9.85,-0.2553125)(9.85,-0.3053125)(9.85,-0.3303125)
\curveto(9.85,-0.3553125)(9.86,-0.3953125)(9.87,-0.4103125)
\curveto(9.88,-0.4253125)(9.905,-0.4503125)(9.92,-0.4603125)
\curveto(9.935,-0.4703125)(9.97,-0.4853125)(9.99,-0.4903125)
\curveto(10.01,-0.4953125)(10.05,-0.5053125)(10.07,-0.5103125)
\curveto(10.09,-0.5153125)(10.125,-0.5303125)(10.14,-0.5403125)
\curveto(10.155,-0.5503125)(10.195,-0.5603125)(10.22,-0.5603125)
\curveto(10.245,-0.5603125)(10.295,-0.5653125)(10.32,-0.5703125)
\curveto(10.345,-0.5753125)(10.395,-0.5803125)(10.42,-0.5803125)
\curveto(10.445,-0.5803125)(10.495,-0.5753125)(10.52,-0.5703125)
\curveto(10.545,-0.5653125)(10.59,-0.5503125)(10.61,-0.5403125)
\curveto(10.63,-0.5303125)(10.66,-0.5003125)(10.67,-0.4803125)
\curveto(10.68,-0.4603125)(10.72,-0.4303125)(10.75,-0.4203125)
\curveto(10.78,-0.4103125)(10.835,-0.3903125)(10.86,-0.3803125)
\curveto(10.885,-0.3703125)(10.915,-0.3403125)(10.92,-0.3203125)
\curveto(10.925,-0.3003125)(10.93,-0.2553125)(10.93,-0.2303125)
\curveto(10.93,-0.2053125)(10.93,-0.1553125)(10.93,-0.1303125)
\curveto(10.93,-0.1053125)(10.935,-0.0603125)(10.94,-0.0403125)
\curveto(10.945,-0.0203125)(10.94,0.0196875)(10.93,0.0396875)
\curveto(10.92,0.0596875)(10.9,0.0996875)(10.89,0.1196875)
\curveto(10.88,0.1396875)(10.85,0.1696875)(10.83,0.1796875)
\curveto(10.81,0.1896875)(10.775,0.2096875)(10.76,0.2196875)
\curveto(10.745,0.2296875)(10.715,0.2546875)(10.7,0.2696875)
\curveto(10.685,0.2846875)(10.65,0.3046875)(10.63,0.3096875)
\curveto(10.61,0.3146875)(10.565,0.3146875)(10.54,0.3096875)
\curveto(10.515,0.3046875)(10.495,0.2996875)(10.5,0.2996875)
\curveto(10.505,0.2996875)(10.495,0.2996875)(10.45,0.2996875)
}
\pscircle[linewidth=0.02,dimen=outer,doubleline=true,doublesep=0.12,doublecolor=lightgray,fillstyle=solid](5.09,-0.1603125){0.6}
\psbezier[linewidth=0.04,arrowsize=0.05291667cm 2.0,arrowlength=1.4,arrowinset=0.4]{->}(1.37,0.18746702)(1.935215,0.7844494)(3.6139464,0.7293253)(4.45,0.0396875)
\usefont{T1}{ppl}{m}{n}
\rput(1.5392188,-0.45){$R_r^k\subset V_l$}
\usefont{T1}{ptm}{m}{n}
\rput(2.6114063,1.0496875){$f$}
\psframe[linewidth=0.04,dimen=outer,fillstyle=solid,fillcolor=lightgray](1.37,0.1196875)(1.09,-0.1603125)
\psdots[dotsize=0.12](5.09,-0.1603125)
\usefont{T1}{ptm}{m}{n}
\rput(5.2828126,-0.2003125){\small 0}
\psbezier[linewidth=0.04](5.69,-0.9603125)(6.69,-0.5603125)(5.89,-0.5603125)(6.49,-0.1603125)(7.09,0.2396875)(5.364269,-0.010963309)(6.09,0.6396875)(6.815731,1.2903383)(6.69,1.0396875)(5.69,0.8396875)(4.69,0.6396875)(5.647092,1.3294717)(4.69,0.8396875)(3.732908,0.34990335)(4.09,0.8396875)(4.09,0.4396875)(4.09,0.0396875)(3.29,0.0396875)(3.89,-0.3603125)(4.49,-0.7603125)(3.49,-0.9603125)(4.29,-1.1603125)(5.09,-1.3603125)(4.69,-1.3603125)(5.69,-0.9603125)
\usefont{T1}{ptm}{m}{n}
\rput(6.0914063,-1.2703125){$B_0$}
\psbezier[linewidth=0.04,arrowsize=0.05291667cm 2.0,arrowlength=1.4,arrowinset=0.4]{->}(6.45,0.5596875)(7.128992,1.0444494)(8.24565,1.0596875)(9.25,0.4796875)
\usefont{T1}{ptm}{m}{n}
\rput(7.6514063,1.3296875){$f^l$}
\usefont{T1}{ptm}{m}{n}
\rput(11.421406,-1.4903125){$B_l$}
\pscircle[linewidth=0.04,dimen=outer](10.49,-0.1603125){1.2}
\psdots[dotsize=0.12](10.49,-0.1603125)
\psframe[linewidth=0.03,dimen=outer](11.055,0.4446875)(10.685,0.0746875)
\psbezier[linewidth=0.04,arrowsize=0.05291667cm 2.0,arrowlength=1.4,arrowinset=0.4]{->}(11.17,0.34355846)(11.871254,0.84395784)(12.817966,0.8596875)(13.69,0.2196875)
\usefont{T1}{ptm}{m}{n}
\rput(12.531406,1.0096875){$f$}
\psline[linewidth=0.02cm,arrowsize=0.05291667cm 2.0,arrowlength=1.4,arrowinset=0.4]{->}(10.89,1.2396874)(10.89,0.5396875)
\usefont{T1}{ptm}{m}{n}
\rput(10.9114065,1.5296875){$R^{k'}_{r'}$}
\end{pspicture} 
}

\caption{Defining $F$ in $W_M^-$}
\end{figure}
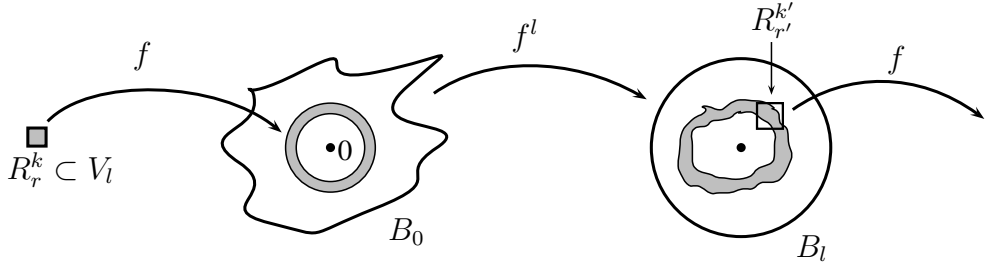

The following lemma can be proved similarly  as Lemma~\ref{4.1}.
\begin{lem}\label{lem13} Put $M=[\alpha_{l_0}]+1$. If $l_0$ is sufficiently large then, for all $l\ge l_0$,
$$F(\Lambda_W\cap  V_l)\subset \Lambda_W\cap Y_M.$$

\end{lem}
\begin{proof}
It is obvious that $F(\Lambda_W\cap V_l)\subset \Lambda_W$. 
Next, $f^{l+1}(V_l)\subset B_l$, so taking $w\in \Lambda_W\cap V_l$, we have for   $f^{l+1}(w)=z\in f^{l+1}(V_l)$:
$$\Re(z)\ge \Re\beta_l-D|\beta_l|\ge\alpha_l-\frac{1}{4}\alpha_l=
\frac{3}{4}\alpha_l$$

Thus,
$$|F(w)|=|f(f^{l+1})(w)|\ge e^{\frac{1}{2}|\alpha_l|}$$
if $l$ is large enough.
As in the proof of Lemma ~\ref{4.1} we conclude that 

$$|\Re F(w)|\ge \frac{1}{K}|F(w)|-\frac{1}{K}\ge \frac{1}{K}e^{\frac{1}{2}\alpha_l}-\frac{1}{K}>[\alpha_l]+1\ge[\alpha_{l_0}]+1$$
for all  $l\ge l_0$, if $l_0$ is large enough.
Therefore, $F(w)\in Y_M$.

\end{proof}

Now, we turn to the proof of the estimate  (\ref{maineq}) (for negative $r$).
Let us fix  some $\delta\in (0,1)$. Let, as in Section~\ref{ppositive}, $\delta'=\delta/2$.
Let $l_0$ be so large that  the statement of Lemma~\ref{lemri} holds with   $M$ replaced by $\frac{1}{2}[\alpha_{l_0}]$. 

Let $R^k_r$ be a rectangle in $V_l$. Then  $f^{l+1}$ maps $R^k_r$ bijectively onto its image, 
contained in $B_l$, and for $z\in R^k_r$ we have:

\begin{equation}\label{pierwszykrok}
|(f^{l+1})'(z)|=|f'(z)||(f^l)'(f(z))|=|f(z)||(f^l)'(f(z))|\ge\frac{\frac{1}{4e}D}{|\beta_1\beta_2\dots \beta_l|}\cdot L^{-1}|\beta_1\beta_2\dots \beta_l|=\frac{D}{4eL}
\end{equation}

The whole set $B_l$ is covered by $O(|\beta_l|)^2)$ rectangles $R^{k'}_{r'}$ of the initial partition (\ref{initial}).

On each such rectangle  $R^{k'}_{r'}$ the map $f$ is a bijection onto its image. Thus, as in Section~\ref{ppositive}, there a family $\mathcal{F}^{k'}_{r'}$ of disjoint sets $Q^{k'}_{r'}$, such that  the set $R^{k'}_{r'}\cap f^{-1}(W_M)$ can be written as a union of the sets  $Q^{k'}_{r'}\in\mathcal{F}^{k'}_{r'}$. Each set $Q^{k'}_{r'}$ is defined as the intersection $f^{-1}(\widehat R^m_s)\cap R^{k'}_{r'}$ where $\widehat R^m_s$ is some rectangle from the family $Z_M$.  Obviously, the inverse map $f^{-1}$ is well defined and holomorphic  
in a neighbourhood of every rectangle $\widehat R^m_s$.

Moreover, since the ball $B_l$ is contained in the set $\{\Re(z)>\frac{1}{2}\alpha_l\}$, for each such rectangle $R^{k'}_{r'}$ Lemma~\ref{lemri} tells  that
\begin{equation}\label{drugikrok}
\sum_{Q^{k'}_{r'}\in\mathcal{F}^{k'}_{r'}}\sup_{z\in Q^{k'}_{r'}}|f'(z)|^{-(1+\delta)}<\hat C e^{-\delta'\cdot \frac{1}{2}\alpha_l}
\end{equation} 

Now, recall that $f^{l+1}$ maps the rectangle $R^k_r$ bijectively onto its image, contained in $B_l$, see Figure 3.  Let $g= (f^{l+1}{|_{R^k_r}})^{-1}$
 be the inverse map.

Taking all the preimages $g(Q)$, $Q\in \mathcal{F}^{k'}_{r'}$ over all rectangles $R^{k'}_{r'}$ intersecting $B_l$, we obtain the family $\mathcal{F}^k_r$ of disjoint subsets of the rectangle $R^k_r$, such that 
the union $\bigcup_{Q\in \mathcal{F}_r^k} Q$ covers the whole set $R^k_r\cap F^{-1}(W_M)$ and each set 
$Q\in\mathcal{F}^k_r$ is mapped by $F=f^{l+2}$ bijectively onto its image, contained in some rectangle 
$\widehat R^m_s\in Z_M$.

Denote by $\mathcal G_l$ the family of all rectangles $R^{k'}_{r'}$ intersecting the ball $B_l$.
Using (\ref{pierwszykrok}) and (\ref{drugikrok}), we get the following estimate.
\begin{equation}
\aligned
&\sum_{Q\in\mathcal{F}^k_r}\sup_{w\in Q}|F'(w)|^{-(1+\delta)}\le\\
&\le \sup_{w\in R^k_r}|(f^{l+1})'(w)|^{-(1+\delta)}\cdot\sum_{R^{k'}_{r'}\in\mathcal{G}_l}\left (\sum_{Q^{k'}_{r'}\in\mathcal{F}^{k'}_{r'}}\sup_{z\in Q^{k'}_{r'}}|(f')(z)|^{-(1+\delta)}\right )\\
&\le (4eLD^{-1})^{(1+\delta)}(\# (\mathcal{G}_l)\cdot \hat Ce^{-\delta'\frac{1}{2}\alpha_l}
\endaligned
\end{equation}

Since the cardinality  $\# (\mathcal{G}_l)$ can be estimated  by $\mathcal O(|\beta_l|^2)$, we can write

$$\sum_{Q\in\mathcal{F}^k_r}\sup_{w\in Q}|F'(w)|^{-(1+\delta)}\le C_1 |\beta_l|^2e^{-\frac{1}{2}\delta'\alpha_l}=C_1|\beta_l|^2|\lambda|^{\frac{1}{2}\delta'}|\beta_{l+1}|^{-\frac{1}{2}\delta'} $$
where $C_1$ is another constant.
Using the "super-growing" condition we easily conclude that
the right-hand side of the above inequality can be estimated from above by
$$|\beta_{l+1}|^{-\frac{\delta'}{4}},$$
for all $l\ge l_0$, if $l_0$ is large enough. 
So, finally we get, for every rectangle $R^k_r$ intersecting $V_l$,

\begin{equation}\label{negative}
\sum_{Q\in\mathcal{F}^k_r}\sup_{w\in Q}|F'(w)|^{-(1+\delta)}\le  |\beta_{l+1}|^{-\frac{\delta'}{4}}<\frac{1}{2}
\end{equation}

This ends  the proof of Proposition~\ref{induk}, with $M=[\alpha_{l_0}]+1$, and   $F$ given by
\begin{equation}
F(z)=\left\{\begin{array}{lcl} f(z)&\mbox{for}& z\in W_M^+  \\ f^{l+2}(z)&\mbox{for}& z\in V_l, {\rm ~~} l\ge l_0+1  \end{array}\right..
\end{equation}
where $l_0$ is chosen large enough to satisfy all the estimates in the present Section (and, consequently, the estimate~(\ref{negative})).
The condition (\ref{pokr}) is satisfied because of Lemma~\ref{4.1}, Lemma ~\ref{lem12} and Lemma~\ref{lem13}. The condition (\ref{maineq}) is satisfied because of the estimates     
(\ref{positive}) and (\ref{negative}).
\begin{xrem}
Recall  that $W_M^-\subset \bigcup_{l=l_0}^\infty V_l$, so the map  $F$ is defined everywhere in $W_M$.
However, since the sets $V_l$ are not  disjoint, there are rectangles $R_r^k$ that are included in both $V_l$ and $V_{l+1}$. In such case we define the map $F$ on the rectangle  $R_r^k$, chosing arbitrarily one of two possible ways.
\end{xrem}

\section{Estimates of the dimension}\label{conclusion}

In this section, we finish the proof of Theorem~\ref{bigM}.
Fix an arbitrary $\delta>0$ and let $M$ be so large that the statement of Proposition~\ref{induk} holds true.

We shall define inductively, for every  rectangle $R^k_r\in Z_M$, and for every $n$, a cover $(\mathcal{F}_r^k)^n$  of the set $\Lambda_W\cap R^k_r$,
such that
\begin{equation}\label{cond}
\sum_{K\in(\mathcal{F}^k_r)^n}({\rm diam} K)^{1+\delta}<(2\pi+1)\cdot\frac{1}{2^n}.
\end{equation}

The cover $(\mathcal{F}^k_r)^1$ is defined simply as the family of sets $\mathcal{F}^k_r$  described in Proposition~\ref{induk}. 
The condition (\ref{cond}) is satisfied since
$$\sum_{Q\in\mathcal{F}^k_r}({\rm diam Q})^{1+\delta}\le (\sum_{Q\in\mathcal{F}^k_r}\sup_{w\in Q}|F'(w)|^{-(1+\delta)})\cdot (2\pi+1)\le\frac{1}{2}\cdot (2\pi+1).$$
Assume that the covers $(\mathcal{F}_r^k)^{n-1}$
have been already defined for every rectangle $R^k_r\in Z_M$.
So, fix some rectangle $R^k_r\in Z_M$, and, next,  some set $Q^k_r\in\mathcal F^k_r$.
It then follows from the construction that 
the  set $Q_r^k$ is mapped by $F$ bijectively onto 
its image $F(Q_r^k)$, which is contained in some rectangle $\widehat{R}^m_s\in Z_M$.

By the inductive assumption, for the  rectangle $\widehat R^m_s\in Z_M$  there is a cover $(\widehat {\mathcal{F}}^m_s)^{n-1}$ of the set  $\widehat R^m_s\cap \Lambda_W$, such that

$$\sum_{\widehat K\in (\widehat{\mathcal F}^m_s)^{n-1}}({\rm diam}(\widehat K)^{1+\delta}<(2\pi+1)\frac{1}{2^{n-1}}.$$

This allows us to define a cover $\mathcal{Q}^k_r$  of the set $Q^k_r\cap \Lambda_W$ as the family of all sets  of the form $F^{-1}_*(\widehat K)$
where $\widehat K\in (\widehat{\mathcal F}^m_s)^{n-1}$ and $F^{-1}_*$ is the inverse of the bijective map
$F:Q^k_r\to F(Q_r^k)$.

Obviously,

$$\sum_{\widehat K\in(\widehat{\mathcal{F}}^m_s)^{n-1}}{\rm diam}(F^{-1}_*(\widehat K))^{1+\delta}\le (2\pi+1)\frac{1}{2^{n-1}}\cdot \sup_{w\in Q^k_r} |F'(w)|^{-(1+\delta)}$$ 

Finally, we define the family  $(\mathcal{F}^k_r)^n$ as the union (over all sets $Q^k_r\in \mathcal F^k_r$ ) of all the covers  $\mathcal{Q}^k_r$,  described above. Since $\mathcal Q^k_r$ is a cover of $Q^k_r\cap \Lambda_W$, we obtain, by taking the union,   a cover of $R^k_r\cap \Lambda_W$.

Moreover, by (\ref{maineq}),

\begin{equation}
\sum_{K\in(\mathcal F^k_r)^n}({\rm diam}K)^{1+\delta}\le \sum_{Q^k_r\in\mathcal F^k_r}\sup_{w\in Q^k_r}|F'(w)|^{-(1+\delta)}\cdot (2\pi+1)\frac{1}{2^{n-1}}\le \frac{1}{2}\cdot (2\pi+1)\frac{1}{2^{n-1}}=(2\pi+1)\frac{1}{2^n}
\end{equation}

This ends the proof of the inequality $\dim_H(\Lambda_W\cap R^k_r)\le 1+\delta$. The conclusion
$\dim_H(\Lambda_W\cap Y_M)\le 1+\delta$ is immediate. Theorem~\ref{bigM} is proved.

\section{Application-Hausdorff dimension of indecomposable continua}\label{wnioski}
In this Section we apply Theorem~\ref{bigM} to show that the Hausdorff dimension of several indecomposable continua, described in Section~\ref{intro}, is equal to one.

The strategy is the following. Since every non-trivial  continuum in the plane has Hausdorff dimension at least one, it is enough to prove the upper bound on the dimension.
We shall use Theorem~\ref{bigM}, and, more precisely, Corollary~\ref{ubd}.
In order to make use of Corollary~\ref{ubd}, we shall check that our continuum, minus  at most one point, is contained
in the  set $\Lambda_{W,ubd}$ for some thin set $W$.
This  has already observed in several particular cases
(see \cite{devaneyjarque}, \cite{devaneyetds}, and \cite{forsterrempeschleicher}) for the most general statement).

For the completeness, we outline the proof of  Lemma~\ref{code} below. This is a particular case, needed in the present paper. We start with a standard fact. See e.g. \cite{urbanskizduniknonhyperbolic}, Lemma 2.2  for its proof.

\begin{lem}\label{standard}
Fix some parameter $\lambda$ for which  the singular trajectory $f^n_\lambda(0)$ escapes (i.e. $f^n_\lambda(z)\to\infty$).
Fix some $R>0$, and denote  by $F_R$ the set of points $z\in\mathbb{C}$ for which the whole trajectory $f^n_\lambda(z)$ remains in the closed ball $\overline  B(0, R)$. 
Then the map ${f_\lambda}_{|F_R}$ is expanding: there exist $c>0$ and $\gamma>1$ such that for every $z\in F_R$
$$|(f^n_\lambda)'(z)|\ge c\gamma^n.$$
\end{lem}

We now assume that the parameter $\lambda$ is chosen so that  the singular value $0$ escapes sufficiently fast. More precisely,  it follows from the classification of the escaping points (\cite{schleicherzimmer} Corollary 6.9) that the fact that $f^n_\lambda(0)$ escapes to $\infty$ implies that   there is a dynamic ray $g_{\underline r}$ such that $0=g_{\underline r}(t_0)$ for some $t_0\ge t_{\underline r}$. 
In Lemma~\ref{code} below, we shall assume additionally that $t_0>t_{\underline r}$, i.e. that the point $0$ is located on an (open) ray.

We  need the following.

\begin{lem}\label{code}
Fix some parameter $\lambda$, such that   the singular trajectory $f^n_\lambda(0)$ escapes, and $0=g_{ \underline r}(t_0)$ for some  $t_0>t_{\underline r}$.  
Let $g_{\underline s}:(t_{\underline s},\infty)$ be a dynamic ray. Denote by $G_s$ the closure of the set $\{g_{\underline s}(t),t>t_{\underline s}\}$ in $\mathbb{C}$. 
Then at most one point in $G_{\underline s}$ has bounded trajectory.
\end{lem}
\begin{proof} In the outline below we use notation from \cite{schleicherzimmer} and \cite{rempenonlanding}.
 To simplify the notation, we write below $g_{\underline r}$ to denote the arc $\{g_{\underline r}(t), t\ge t_0\}$. Put $f=f_\lambda$.
Now, it is convenient to use a "dynamical coding", as proposed in \cite{rempenonlanding}, Section 2.
The preimage $f^{-1}(g_{\underline r})$  is a union of countably many curves $g_{k\underline r}$, which   cut  the plane, defining   countably many open strips 
$S_k$, where  $S_k$ is bounded by $g_{k\underline r}$ and $g_{(k+1)\underline r}$. Denote also by $\overline S_k$ the closure of $S_k$.

Each strip $S_k$ is mapped by $f$ univalently onto $\mathbb{C}\setminus g_{\underline r}$. Denote by $f^{-1}_k$ the inverse map: $f^{-1}_k:\mathbb{C}\setminus g_{\underline r}\to S_k$.

Let $g_{\underline s}$ be an arbitrary  dynamic ray. 
It is easy to check that the ray $g_{\underline s}$ does not  intersect the ray $g_{k\underline r}$ unless $g_{\underline s}=g_{ k \underline r}$. Thus, all the points on the ray $g_{\underline s}$ share the same "dynamic address": there is a sequence $\underline{\tilde s}=\tilde s_0, \tilde s_1\dots$ such that for every point $z\in g_{\underline s}$  and every $n$ we have  $f^n (z)\in S_{\tilde s_n}$.         
Obviously, for every $z\in G_{\underline s}$ we have  $f^n (z)\in \overline S_{\tilde s_n}$.         

\

Now, let us fix some $R>0$, and consider the set  $F_R$ defined in Lemma~\ref{standard}.
Since all points in the rays $g_{k\underline r}$ escape to $\infty$
, and the trajectory of points $z\in F_R$ remains bounded, there exists $\delta>0$ such that the trajectory of every point  
$z\in F_R$ is $\delta$-separated 
from the boundary curves of all strips $S_k$, and, by the same reason, 
from the curves $f^k(g_{\underline 
r})$. 
(Here we use the fact that $f^k(z)\to\infty$ uniformly on the curve $g_{\underline r}:=g_{\underline r} 
([t_0,\infty))$, so, actually, only finitely many of curves $f^k(g_{\underline r})$ intersect the ball 
$\overline 
B(0,R)$.)

This easily implies the following: there exists a constant $L>0$ such that, for arbitrary two points $z,w\in F_R$, belonging to the same strip, there exist a topological disc $U=U_{z,w}$ on which each  composition
$$f^{-1}_{k_0}\circ f^{-1}_{k_1}\circ\dots \circ f^{-1}_{k_{n-1}}$$
is well defined, with distortion bounded by $L$.

 





Let $x, y\in F_R$. If $x,y \in G_{\underline s}$, then they have the same dynamic code $\tilde{\underline s}(x)=\tilde {\underline s}(y)=\tilde s_0, \tilde s_1, \dots, \tilde s_{n-1},\tilde s_n\dots$. Put $z=f^n(x)$, $w=f^n(y)$ and take the inverse branch following our coding
$$f^{-n}_*=f^{-1}_{\tilde s_0}\circ f^{-1}_{\tilde s_1}\circ \dots \circ f^{-1}_{\tilde s_{n-1}}.$$

Then $f^{-n}_*(z)=x$ and $f^{-n}_*(w)=y$ (because $\tilde{\underline s}(x)=\tilde {\underline s}(y)$).
Expanding property (Lemma~\ref{standard}) together with the bounded distortion property give

$$|x-y|=|f^{-n}_*(z)-f^{-n}_*(w)|\le \frac{L\cdot 2R}{c}\gamma^{-n}.$$
Since $n$ can be taken arbitrarily large, we get $x=y$.


\end{proof}
\
Note that, actually, we proved in Lemma~\ref{code} the following fact:
\begin{lem}\label{code2}
Under the assumption and notation of Lemma~\ref{code}, let $\tilde {\underline s}=(\tilde s_0,\tilde s_1,\tilde s_2\dots)$ be a dynamic address and  
let $H_{\tilde{\underline s}}$ be the set of points $z\in\mathbb{C}$ such that, for every $n\in\mathbb{N}$, $f^n_\lambda(z)\in\overline S_{\tilde s_n}$. 
Then at most one point in $H_{\tilde{\underline s}}$ has bounded trajectory.
\end{lem}  
We are ready to formulate the following corollaries.
 
\begin{twr}\label{contdevaney}
Let $f(z)=\exp(z)$ and let $\Lambda$ be the indecomposable continuum described in \cite{devaneyetds}:
$$\Lambda=\{z\in \mathbb{R}:\forall n\ge 0~~ {\rm Im} (f^n z)\in [0,\pi]\}$$
Then $\dim_H(\Lambda)=1$.
\end{twr}
\begin{proof}
This is an immediate consequence of Theorem~\ref{bigM} and  Lemma~\ref{code2}:
Putting 
$$W=\{z:{\rm Im}(z)\in [0,\pi]\}$$
we see immediately that $\Lambda$ is contained in  the set of points whose trajectory remains in the closed dynamic strip  $\overline S_0=\mathbb{R}\times [0,2\pi]$.
(Note, besides,  that  points from $S_0\setminus W$ will leave $S_0$ immediately.)
Thus, using  using Lemma~\ref{code2} we see  that ${\rm {card}}(\Lambda\setminus  \Lambda_{W,ubd})\le 1$. 
Therefore, by Theorem~\ref{bigM}, $\dim_H(\Lambda)\le 1$. Since $\Lambda$ is a non-trivial continuum, $\dim_H(\Lambda)\ge 1$. This gives the required equality.
\end{proof}
\begin{rem}
The inclusion ${\rm {card}}(\Lambda\setminus  \Lambda_{W,ubd})\le 1$ can
be also deduced from the detailed  description of $\omega$- limit sets of the points in $\Lambda$, provided in \cite{devaneyetds}.
\end{rem} 
Now, let us consider a more general situation:
Assume that the parameter $\lambda$ satisfies $f^n_\lambda(0)\to\infty$.  In particular, this implies that the singular value $0$ is on a dynamic ray or is the landing point of such a ray.  Denote this ray, as in the proof of Lemma~\ref{code},  by $g_{\underline r}:(t_{\underline r},\infty)\to\mathbb{C}$. For such maps, 
L. Rempe provides the construction of uncountably many dynamically defined indecomposable continua. Each such a  continuum  is defined as the accumulation set of some dynamic ray (see \cite{rempenonlanding}, Theorem 1.2). Namely,  one considers the set $R_1$ of external addresses  of the following form:

$$\underline s = s(n_1, n_2, n_3,\dots )
:= T_1r_0r_1 . . . r_{n_1-1}T_{n_1}r_0r_1 . . . r_{n_2-1}T_{n_2}r_0r_1 . . . r_{n_3-1}T_{n_3}\dots$$
where $T_n := 2 + \max
_{k\le n}r_k$. A suitably chosen, uncountable subset $R\subset R_1$ has the required property: for every ray $\underline s\in R$ the accumulation set of the ray $g_{\underline s}$ is an indecomposable continuum.  

Before formulating the  result about the continua described in \cite{rempenonlanding}, we note the following auxiliary fact.

\begin{lem}\label{code3} 
Let $\lambda\in \mathbb{C}\setminus \{0\}$ is chosen so that the trajectory $f^n_\lambda(0)$ escapes to $\infty$.

Let $g_{\underline u}:(t_{\underline u},\infty)\to\mathbb{C}$ be a dynamic ray such that the sequence $u_n$ is bounded. If the ray $g_{\underline u}$ "lands", i.e. if there exists a finite limit $$\lim_{t\searrow t_{\underline u}}g_{\underline u}(t)=a\in\mathbb{C}$$ then
the trajectory $f^n_\lambda(a)$ does not escape to $\infty$.
\end{lem}
\begin{proof}
Let $K=\sup_n\{|u_n|\}$. As mentioned above, since the point $0$ escapes, there is an external ray 
$g_{\underline r}:(t_{\underline r},\infty)\to\mathbb{C}$ such that  $0$ is 
either its landing point or it is contained in $g_{\underline r}(t_{\underline r}, \infty)$. 
Then, as in the  proof of Lemma \ref{code}, and using its notation, we 
consider  the dynamic coding defined with use of the curves $g_{k\underline r}$. Since the rays do not 
intersect, every   image of the ray   $f^n_\lambda( g_{\underline u}([t_{\underline u},\infty)))$,  is 
contained in some strip $S_{k(n)}$, and it is easy to see, using the behaviour of the rays at infinity  (see Prop. 3.4 in \cite{schleicherzimmer})
 that $|k(n)|<K''$ for some $K''\ge K$. Moreover,  there exist $M'>0$ and $K'\ge K''$  such that
$$\bigcup_{|k|<K'}S_k\cap \{{\rm Re}z>M\}\subset \bigcup_{|k|<K'}P_k.$$

 Denote by $\Sigma_{K'}$ the set of all infinite sequences $\underline w=(w_n)_{n=0}^\infty$, with integer entries, and such that $|w_n|<K'$ for all $n$.  
Every point $z\in \mathbb{C}$ 
has its "geometric address" $\underline s(z)$ defined by $s_i(z)=k$ if $f^n_\lambda(z)\in P_k$.  
According to \cite{dk}, Section 3 (see also \cite{schleicherzimmer}, Prop. 3.4), there exists $M>0$ depending on $K'$ and $\lambda$ such that the set of "directly escaping points":
$$E_{M,K'}=\{z\in\mathbb{C}: {\rm Re}f^n_\lambda(z)\ge M{\rm~~ for~ all~~} n\ge 0 {\rm ~and~~}\underline 
s(z)\in\Sigma_{K'}\}$$ 
is a union of "tails of rays" $g_{\underline s}(t)$, $\underline s\in\Sigma_{K'}$.  Each tail is a curve to $\infty$,  and, actually, a graph of a function defined in $\{x\ge M\}$, and for every $z\in g_{\underline s}([M,\infty)$ we have  $\underline s(z)=\underline s$.

Now, let $a$ be the landing point of the tail $g_{\underline u}$, and assume that $f^n_\lambda(a)\to\infty$. Then there exists $n_0$ such that for $n\ge n_0$ ${\rm Re}f^n_\lambda(a)>M+1$.
Consequently, $f^{n_0}_\lambda(a)\in E_{M,K'}$ and, since ${\rm Re} f^{n_0}(a)>M+1$, $f^{n_0}(a)$ is in a tail of some ray (thus, it is located in an open ray). This is a contradiction since $f_\lambda$ maps ends of rays to ends of rays.      
\end{proof}

\begin{twr}\label{contrempe}
Assume that the parameter $\lambda$ satisfies $f^n_\lambda(0)\to\infty$. Assume additionally that ${\rm Im}(f^n_\lambda(0))$ remains bounded. Let $\Lambda$ be the indecomposable continuum constructed in  \cite{rempenonlanding}. Then $\dim_H(\Lambda)=1$.

\end{twr}

\begin{proof}[Proof of Theorem~~\ref{contrempe}]
As mentioned in Section~\ref{statement}, it is easy to see that the assumptions of the Theorem imply that the parameter $\lambda$ satisfies the super-growing condition (\ref{super}).
Since $0\in I(f_\lambda)$, it follows from the classification of escaping points (\cite{schleicherzimmer}) that $0$ belongs to some dynamic ray $g_{\underline r}(t_{\underline r}, \infty)\to\mathbb{C}$, or it is a landing point of a ray. Lemma~\ref{code3} allows us to exclude this second possibility. Therefore, $0=g_{\underline r}(t)$ for some $t>t_{\underline r}$. Open dynamic rays are smooth curves  \cite{viana}, see also \cite{forsterschleicher} for the estimates of the second derivative of the parametrization $t\mapsto g_{\underline r}(t)$, $t\in (t_{\underline r},\infty)$. 
In particular, the ray $g_{\underline r}(t_{\underline r},\infty)$ has a tangent vector at $0$.  Let us consider, as in the proof od Lemma~\ref{code}, the rays $g_{k\underline r}$. 
Then the existence of a tangent vector of $g_{\underline r}$ at $0$ guarantees that  for every $k\in\mathbb{Z}$ there exists a finite limit 
$$\lim_{{\rm Re z}\to -\infty, z\in g_{k\underline r}} {\rm Im}(z)=A_k$$
Moreover, it follows from the general bounds for the parametrization of hairs (see \cite{schleicherzimmer}, Prop 3.4) that there exists a finite limit
$$\lim_{{\rm Re} z\to+\infty, z\in g_{k\underline r}} {\rm Im} z=B_k=2\pi k-Arg\lambda$$
Obviously, $A_{k+1}=A_k+2\pi$, $B_{k+1}=B_k+2\pi$.

Now, let $\underline s$ be an arbitrary external address with bounded entries, and let $g_{\underline s}:
(t_{\underline s},\infty)\to\mathbb{C}$ be the corresponding external ray. Put  $G_{\underline 
s}=\overline{g_{\underline s}}$. As in the proof of Lemma~\ref{code} we notice that $G_{\underline s}$ is 
contained in the closure of some dynamic strip $S_k$ bounded by the curves $g_{k\underline r}$ and 
$g_{(k+1)\underline r}$. 

Moreover, again by the abovementioned Prop.3.4 in \cite{schleicherzimmer} we know  that for every dynamic ray $g_{\underline s}$
$$\lim_{t\to+\infty}{\rm Im}g_s(t)=2\pi s_1-Arg\lambda.$$
It now easily follows from two above observations that for every bounded sequence $g_{\underline s}$ there exists $K>0$ such that 

$$\bigcup_{n\ge 0}f^n_\lambda (G_{\underline s})\subset W=\left\{z\in\mathbb{C}:|{\rm Im}z|\le K\right\}$$  
Thus,
$$\Lambda\subset G_{\underline s}\subset \Lambda_W.$$ 

By Lemma~\ref{code} we have  ${\rm {card}}(G_{\underline s}\setminus  \Lambda_{W,ubd})\le 1$.
Therefore, using Theorem ~\ref{bigM}, we conclude that $\dim_H (\Lambda)\le 1$, and, again, since $\Lambda$ is a non-trivial continuum, $\dim_H(\Lambda)=1$.
\end{proof}
As a corollary we have
\begin{cor}
If $\Lambda$ is an indecomposable continuum constructed, for the map $f_1(z)=\exp(z)$, in \cite{devaneyjarque} then $\dim_H(\Lambda)=1$.
\end{cor}

\section{Hausdorff measure of the continua}\label{hm}
As a complement of our result let us note the following general remark:
\begin{stw}\label{sskoncz}
One--dimensional Hausdorff measure of an indecomposable continuum in the plane is not $\sigma$-finite. 
\end{stw}
To prove Proposition \ref{sskoncz} we will need to state some facts about indecomposable continua. See \cite{kur} for definitions and properties. Let us denote the continuum by $\mathcal{C}$.
\begin{dfn}
For a point $p$ define the set $K_p$ as a union of all proper subcontinua containing $p$, in other words
$K_p=\{x\colon \mbox{$\exists$ a proper subcontinuum $S$ of $\mathcal{C}$ containing $p$ and $x$}\}$. The set $K_p$ is called the \emph{composant} of the  point $p$.
\end{dfn}
The following facts are easy to verify cf. \cite{kur}, Section 48.
\begin{fact} Every composant of an indecomposable metric continuum  is an $F_\sigma$ set of the first category.
\end{fact}

\begin{fact}
In an indecomposable metric continuum every composant is a dense subset of $\mathcal{C}$.
\end{fact}
\begin{fact}
An indecomposable metric continuum is a union of uncountably many disjoint composants.
\end{fact}

The next  fact follows immediately from the   definition of Hausdorff measure.
\begin{fact} \label{fact3}
A connected measurable  set $A\subset \mathbb{R}^2$ with $\mathrm{diam}(A)>0$ has positive 1-dimensional Hausdorff measure $\mathcal{H}_1(A)>0$.
\end{fact}
The above facts show that an indecomposable continuum consists of uncountably many disjoint sets of positive measure $\mathcal{H}_1$. This proves that the whole set cannot be $\sigma$-finite with respect to $\mathcal{H}_1$. The formal proof follows.
\begin{proof}[Proof of Proposition \ref{sskoncz}]
Assume the opposite: the continuum $\mathcal{C}$ can be represented as a union of countably many  disjoint measurable sets of finite measure: \[\mathcal{C}=\bigcup_{i=1}^\infty C_i\mbox{, \;} C_i\cap C_j =\emptyset \mbox{ (for $i\neq j$)} \mbox{ and } \mathcal{H}_1(C_i)<+\infty. \]
The set $\mathcal{C}$ is also a union of disjoint composants of some points: $\mathcal{C}=\bigcup_{p\in P} K_p$, where ${\rm card}(P) >\aleph_0$. 
Define the sets $K_p^i=K_p\cap C_i$.

The set $\mathcal{C}$ has a positive diameter $a$. Since  every $K_p$ is dense in $\mathcal{C}$, the diameter of $K_p$ is bigger than $\frac{1}{2}a$. Using Fact \ref{fact3} we see that $\mathcal{H}_1(K_p)>0$.

As $K_p=\bigcup_i K_p^i$ this means that for every $p$ at least one $K_p^i$ has positive measure $\mathcal{H}_1(K_p^i)>0$. Let us denote  these chosen sets by  $\widehat{K}_p$.

Define $A_n=\{p\in P\colon \mathcal{H}_1(\widehat{K}_p) >\frac{1}{n} \}$. The set $P$ is not countable so there exists $A_n=A_{n_0}$ containing uncountably many $p$'s. We denote them by $\widehat{p}$.

Finally, since there are countably many sets $C_i$,  one of them contains infinitely many of  $\widehat{K}_{\widehat{p}}$'s.
Choosing a countable subset $(\widehat p_k)_{k=1}^\infty$, such that $\widehat K_{\widehat p_k}\subset C_i$, we can write
\[\mathcal{H}_1(C_i)\ge\mathcal{H}_1\left(\bigcup_{k=1}^\infty\widehat{K}_{\widehat{p}_k}\right) = \sum_{k=1}^\infty\mathcal{H}_1\left(\widehat{K}_{\widehat{p}_k}\right )=\infty\]thus giving a contradiction.
\end{proof}



\begin{thebibliography}{99}
\bibitem[Be]{bergweiler} W.Bergweiler: Iteration of meromorphic functions. Bull. AMS 29 (1996) 151-188.
\bibitem[De]{de1} R. Devaney: Cantor Bouquets, Explosions, and Knaster Continua: Dynamics of Complex Exponentials.
Publicacions Matematiques 43 (1999), 27-54.
\bibitem[De2]{devaneyetds} R.Devaney: Knaster-like continua and complex dynamics, Ergodic Theory and Dynamical Systems 13 (1993), 627–634.
\bibitem[DK]{dk} R.Devaney, M. Krych: Dynamics of Exp(z) Ergodic Theory and Dynamical Systems 4 (1984), 35-52.
\bibitem[DJ]{devaneyjarque} R. Devaney, X.Jarque:  Indecomposable Continua in Exponential Dynamics,  Conformal Geometry and Dynamics 6 (2002), 1-12.
\bibitem[DJR]{djr} R.Devaney,  X. Jarque, M. Moreno Rocha: Indecomposable Continua and Misiurewicz Points in Exponential Dynamics
International Journal of Bifurcation and Chaos 15 (2005), 3281-3293.
\bibitem[EL]{eremenkolyubich}Dynamical properties of some classes of entire functions. Ann. Inst. Fourier (Grenoble) 42 (1992), no. 4, 989–1020. 
\bibitem[FRS]{forsterrempeschleicher}  M.F\"orster, L.Rempe, D.Schleicher: Classification of escaping exponential maps, Proc. AMS  136 (2008), 651-663  
\bibitem[FS]{forsterschleicher} M. F\"orster, D.Schleicher: Parameter rays for exponential family, Erg. Th. Dynam. Sys. 29(2009) 515--544  
\bibitem[GK]{goldbergkeen}L. Goldberg, L. Keen: A finiteness theorem for a dynamical class of entire functions. Ergodic Theory Dynam. Systems 6 (1986), no. 2, 183--192. 
\bibitem[K]{kur} K. Kuratowski: Topology, Vol. II, Academic Press, New York, 1968.
\bibitem[Ka]{ka} B. Karpi\'nska: Hausdorff dimension of the hairs without endpoints for
$\lambda \exp (z)$,  C. R. Acad. Sci. Paris, 328, Serie I,  (1999), 1039-1044.
\bibitem[KU]{ku} J. Kotus, M.Urbanski:  Fractal measures and ergodic theory of transcendental meromorphic functions,
London Math. Soc. Lect. Notes. 348 (2008), Volume: Transcendental Dynamics and Complex Analysis.
\bibitem[Lyu]{lyubich} Measurable dynamics of an exponential.(in Russian) Dokl. Akad. Nauk SSSR 292 (1987), no. 6, 1301–1304.
English translation: Soviet Math. Dokl. 35 (1987), no. 1, 223–226.

\bibitem[Mi]{misiurewicz}M. Misiurewicz:  On iterates of $e^z$. Ergodic Theory Dynamical Systems 1 (1981), no. 1, 103--106. 
\bibitem[Re]{rees} M.Rees: The exponential map is not recurrent. Math. Z. 191 (1986), no. 4, 593--598.
\bibitem[Rem]{rempenonlanding} L.Rempe: On nonlanding dynamic rays of exponential maps. Ann. Acad. Sci. Fenn. Math. 32 (2007), no. 2, 353--369.
\bibitem[Rem2]{rempetopo} L. Rempe: Topological dynamics of exponential maps on their escaping sets, Ergod Th. Dynam. Sys. 26 (2006), 1939--1975
\bibitem[Ro]{ro}M. Moreno Rocha: Existence of indecomposable continua for unstable exponentials,
Top Proc, 27 (1): 233-244 (2003)
\bibitem[Sch]{schleicherparadox}D. Schleicher: The dynamical fine structure of iterated cosine maps and a dimension paradox. Duke Mathematics Journal 136 (2007), 343--356.
\bibitem[SchZi]{schleicherzimmer} D. Schleicher, J. Zimmer:  Periodic points and dynamic rays of exponential maps. Ann. Acad. Sci. Fenn. Math. 28 (2003), no. 2, 327–354.
\bibitem[UZ]{urbanskizduniknonhyperbolic}M.Urba\'nski, A.Zdunik: Geometry and ergodic theory of non-hyperbolic exponential maps. Trans. Amer. Math. Soc. 359 (2007), no. 8, 3973–3997. 
\bibitem[V]{viana} M. Viana da Silva: The differentiability of the hairs of $\exp(z)$. Proc. AMS 103(1988), 1179--1184
\bibitem[We]{we} Qiu Weiyuan: Hausdorff dimension of the $M$-set of $\lambda$ Acta Mathematica Sinica 10 no. 4 (1994), 362-368
\end{thebibliography}
\end{document}